\documentclass[10pt]{article}

\usepackage{amssymb,amsthm,amsmath,mathrsfs}
\usepackage{enumerate}
\usepackage{color, graphicx}
\usepackage[hidelinks]{hyperref}

\newcommand{\R}{\mathbb{R}}

\newcommand{\scal}[2]{\left\langle #1, #2 \right\rangle}

\DeclareMathOperator{\conv}{conv}

\usepackage[paper=a4paper, left=1.4in, right=1.4in, top=1.5in, bottom=1.5in]{geometry}
\newtheorem{theorem}{Theorem}
\newtheorem{lemma}[theorem]{Lemma}

\theoremstyle{remark}
\newtheorem{remark}[theorem]{Remark}

\theoremstyle{definition}

\title{\vspace{-3em}A note on the extremal noncentral sections of the cross-polytope}
\author{Ruoyuan Liu\thanks{Carnegie Mellon University; Pittsburgh, PA 15213, USA; Email: ruoyuanl@andrew.cmu.edu}
\ and 
Tomasz Tkocz\thanks{Carnegie Mellon University; Pittsburgh, PA 15213, USA; Email: ttkocz@math.cmu.edu. Research supported in part by the Collaboration Grants from the Simons Foundation.}
}
\date{15th October 2019}

\begin{document}

\maketitle

\begin{abstract}
We find minimal and maximal length of intersections of lines at a fixed distance to the origin with the cross-polytope. We also find maximal volume noncentral sections of the cross-polypote by hyperplanes which are at a fixed large distance to the origin and minimal volume sections by symmetric slabs of a large fixed width. This parallels recent results about noncentral sections of the cube due to Moody, Stone, Zach and Zvavitch.
\end{abstract}

\bigskip

\begin{footnotesize}

\end{footnotesize}

\bigskip

\section{Introduction}
 
The study of sections and projections of various classes of convex bodies is one of the main topics in convex geometry and geometric tomography (see e.g. monographs \cite{Ga, Kol}). Questions of extremal volume sections have received considerable attention in the last few decades, particularly in the case of sections of the $n$-dimensional cube $B_\infty^n = [-1,1]^n$. Hadwiger in \cite{Ha} determined minimal volume sections of $B_\infty^n$ by $n-1$-dimensional subspaces. In response to Good's questions (motivated by geometry of numbers), Hensley in \cite{Hen} developed a  probabilistic approach, reproved Hadwiger's result and established upper bounds. Continuing his ideas, Vaaler in \cite{Vaa} found minimal volume sections of the cube by subspaces of arbitrary dimension. Using Fourier analytic methods, Ball in \cite{Ball-sec} (see also \cite{NP}) solved the question of maximal volume sections of $B_\infty^n$ by subspaces of dimension $n-1$ and extended this later in \cite{Ball-sec2} to all dimensions at least $n/2$ or dividing $n$ (his argument was based on an ingenious use of the Brascamp-Lieb inequality). There are also extensions and related results for the volume replaced by other measures (for instance, see \cite{KK1, KK2, BGMN, Z1, Z2}). Barthe and Kodobsky in \cite{BK} studied extremal volume sections of the cube by symmetric slabs (see also K\"onig and Koldobsky's follow-up \cite{KK0} furnishing a complete solution for $2$ and $3$-dimensional cubes which supports V. Milman's conjecutre). We also mention a very recent paper \cite{KR} by K\"onig and Rudelson about noncentral sections of the cube where they establish that $\min |B_\infty^n \cap H|$, the minimum being over all $n-k$ dimensional affine subspaces at distance at most $\frac{1}{2}$ from the origin, is lower bounded by a positive constant which depends only on $k$. Another direction of extending results for the cube is by looking at sections of $\ell_p$ balls, that is the sets $B_p^n = \{x \in \R^n, \sum_{i=1}^n |x_i|^p \leq 1\}$ with $p \in [1,\infty]$. When $p=1$, $B_1^n$ is the cross-polytope, that is $B_1^n = \conv\{\pm e_1, \ldots, \pm e_n\}$, where the vectors $e_1, \dots, e_n$ denote the standard basis vectors in $\R^n$ ($e_j$ having $1$ at $j$th coordinate and $0$ elsewhere). Extending Vaaler's approach, Meyer and Pajor in \cite{MeyP} showed that for every subspace $H$ of $\R^n$, the function $[1,\infty] \ni p \mapsto \frac{|B_p^n \cap H|}{|B_p^k|}$ increases (here and throughout, $|\cdot|$ denotes Lebesgue measure). Since for $p=2$, the value of this function is $1$ regardless $H$, this gives that the maximal volume sections of $B_p^n$ for $p \in [1,2]$ and minimal ones for $p \in [2,\infty]$ are by attained by coordinate subspaces. There are also several results for projections (see for instance Barthe and Naor's work \cite{BN} or a recent paper \cite{Iv} by Ivanov). Along the way, important probabilistic tools have been discovered to work with the uniform measure on $B_p^n$, notably its probabilistic represenatation from \cite{BGMN} due to Barthe, Gu\'edon, Mendelson and Naor. A unified approach to sections and projections for central hyperplanes has recently been developed in \cite{ENT1}. There are also analogues and extensions to complex setting (see, e.g. \cite{KK-C, KolZ, OP}).

Finding extremal volume sections by noncentral subspaces pose of course additional challenges. Recently Moody, Stone, Zach and Zvavitch in \cite{zva} determined minimal and maximal length of intersections of lines at a fixed distance to the origin with the $n$-dimensional cube $B_\infty^n$. They also established maximal volume noncentral sections of the cube by hyperplanes which are at a fixed distance $t$ to the origin when $t > \sqrt{n-1}$ and minimal volume sections by symmetric slabs of a fixed width $2t$. 

The aim of this paper is to parallel these results for the cross-polytope. We largely follow their direct approach of reducing the whole problem to low-dimensional explicit optimisation questions (either by a case analysis based on combinatorially limited extreme situations, or projection type arguments), but of course the technical details are somewhat different. We present our results in the next section, which is then followed by the section containing all proofs. We use $\scal{x}{y} = \sum_{i=1}^n x_iy_i$ to denote the standard scalar product of two vectors $x$, $y$ in $\R^n$, $|x| = \sqrt{\scal{x}{x}}$ to denote the length of $x$, $\conv A$ to denote the convex hull of a set $A$ in $\R^n$ and $[x,y] = \conv\{x,y\}$ to denote the segment in $\R^n$ with endpoints $x$, $y$.


\section{Results}\label{sec:res}

The maximal length of noncentral sections by lines are attained at $2$-dimensional sections by coordinate subspaces, which is the content of our first result (cf. Theorem 1 from \cite{zva}). Note that the length of the maximal section given by \eqref{eq:max-1-dim} below as a function of $t$, the distance of lines to the origin, is discontinuous at one point, namely $t = \frac{3}{4}$. 

\begin{theorem}\label{thm:max-1-dim}
For every $t \in [0,1]$, let $\mathscr{L}_t$ be the set of lines in $\R^n$ which are at distance $t$ away from the origin. For $n \geq 2$, we have 
\begin{equation}\label{eq:max-1-dim}
\max_{\ell \in \mathscr{L}_t} |B_1^n \cap \ell| = \begin{cases} \frac{2}{t+\sqrt{1-t^2}}, & t \in [0,\frac{1}{\sqrt{2}}], \\
t- \sqrt{t^2-\frac{1}{2}}, & t \in (\frac{1}{\sqrt{2}},\frac{3}{4}], \\
2 - 2t, & t \in (\frac{3}{4},1]. \end{cases}
\end{equation}
The maximum is attained if $\ell$ is contained in a $2$-dimensional coordinate subspace, that is spanned by $e_i$ and $e_j$ for some $i \neq j$.
\end{theorem}

The minimal length noncentral sections by lines are described in the next result. We point out that for the cube, the minimal sections are attained at $2$ dimensional sections and for the maximal sections the answer breaks into $n$ cases according to the value of $t$ (see Theorems 2 and 4 in \cite{zva}), whereas for the cross-polytope -- the other way around.

\begin{theorem}\label{thm:mim-1-dim}
For every $t \in [0,1]$, let $\mathscr{L}_t$ be the set of lines in $\R^n$ which are at distance $t$ away from the origin. Let $T_n(k) = \frac{\sqrt{(k+1)(n-k)}+\sqrt{k(n-k-1)}}{n(\sqrt{k}+\sqrt{k+1})}$, $k = 0, 1,\dots,n-1$. Then $\frac{1}{\sqrt{n}} = T_n(0) > T_n(1) > T_n(2) > \dots > T_n(n-1)$ and for $n \geq 2$, we have 
\begin{equation}\label{eq:min-1-dim}
\min_{\ell \in \mathscr{L}_t} |B_1^n \cap \ell| = \begin{cases} \frac{2}{\sqrt{n}}, & t \in [0,T_n(n-1)], \\
2\frac{1-t\sqrt{n-k}}{\sqrt{k}}, & t \in [T_n(k),T_n(k-1)], \ k = n-1,\dots, 1, \\
0, & t \in (\frac{1}{\sqrt{n}},1]. \end{cases}
\end{equation}
\end{theorem}

\begin{remark}
It will be clear from the proof that for $t \in [0,T_n(n-1)]$, the minimum in \eqref{eq:min-1-dim} is attained by a line passing through two parallel facets of $B_1^n$ perpendicular to them and for $t \in [T_n(k),T_n(k-1)]$, the minimum is attained by a line connecting points $(1-t\sqrt{n-k})\frac{\sum_{i=1}^k e_i}{k} + t\sqrt{n-k}\frac{\sum_{i=k+1}^n e_i}{n-k}$ and $-(1-t\sqrt{n-k})\frac{\sum_{i=1}^k e_i}{k} + t\sqrt{n-k}\frac{\sum_{i=k+1}^n e_i}{n-k}$. For $t \in (\frac{1}{\sqrt{n}},1]$, the minimum is plainly attained by lines disjoint from $B_1^n$. As previously, the length of the minimal section given by \eqref{eq:min-1-dim} as a function of $t$, the distance of lines to the origin, is discontinuous at one point, namely $t = \frac{1}{\sqrt{n}}$. 
\end{remark}

We are also able to identify maximal hyperplane sections by hyperplanes which are at a fixed large distance from the origin, large meaning here at least $\frac{1}{\sqrt{2}}$ (this forces the hyperplanes to separate exactly one vertex).

\begin{theorem}\label{thm:max-hyp}
For every $t \in [0,1]$, let $\mathscr{H}_t$ be the set of hyperplanes in $\R^n$ which are at distance $t$ away from the origin. For $n \geq 3$ and $t \in (\frac{1}{\sqrt{2}},1]$, we have
\begin{equation}\label{eq:max-hyp}
\max_{H \in \mathscr{H}_t} |B_1^n \cap H| = \frac{2^{n-1}(1-t)^{n-1}}{(n-1)!}.
\end{equation}
The maximum is attained if and only if $H$ is parallel to one of the coordinate hyperplanes $\{x \in \R^n, \ x_i = 0\}$.
\end{theorem}

As a corollary, we can find the minimal volume sections by slabs of large width.

\begin{theorem}\label{thm:min-slab}
For $n \geq 3$, $t \in (\frac{1}{\sqrt{2}},1]$ and every unit vector $a$ in $\R^n$, we have
\begin{equation}\label{eq:min-slab}
|B_1^n \cap \{x \in \R^n, \ |\scal{x}{a}| \leq t\}| \geq \frac{2^n}{n!}(1 - (1-t)^n)
\end{equation}
with equality if and only if the direction of $a$ is along one of the axes.
\end{theorem}

\section{Proofs}\label{sec:proofs}

We now turn to proofs of our results. For the maximal sections by lines, following \cite{zva}, we crucially use the fact that the extreme lines have to pass through edges, which greatly simplifies the problem reducing it to questions in $2$, $3$ and $4$ dimensions (depending on the type of edges involved). For the minimal sections by lines, we use a projection argument inspired by one from \cite{zva} (in the cube case the product structure is exploited, whereas we end up with an explicit $2$-dimensional problem to solve -- see Lemma \ref{lm:isosceles}). For the sections by hyperplanes, we first derive a formula for the volume of a ``chopped-off'' pyramid (by breaking it into simplices) and then solve an optimisation problem in $\R^n$. We will also need the following basic fact which can be checked by a direct computation.

\begin{lemma}\label{lm:dist(l,0)}
Let $\ell$ be a line in $\R^n$ passing through two distinct points $a$ and $b$. Then $\ell$ is at distance $\frac{\sqrt{|a|^2|b|^2-\scal{a}{b}^2}}{|a-b|}$ from the origin.
\end{lemma}

\subsection{Maximal $1$-dimensional sections: Proof of Theorem \ref{thm:max-1-dim}}

Every edge of $B_1^n$ is a segment of the form $[\pm e_i, \pm e_j]$ for some $i \neq j$ and a choice of signs.
Fix $t \in [0,1]$. By Lemma 2 from \cite{zva}, a line $\ell$ for which the maximum in \eqref{eq:max-1-dim} is attained passes through two edges of $B_1^n$. Say these edges are segments $A = [\pm e_i, \pm e_j]$ and $B = [\pm e_k, \pm e_l]$. We have $4$ possibilities depending on the number of distinct indices among $i, j, k, l$.

\bigskip
\noindent
\emph{Case 1:} $|\{i, j, k, l \}| = 2$. By symmetry we can assume $i = k = 1$ and $j = l =2$, that is the maximal segment $B_1^n \cap \ell$ has endpoints on $A = [\pm e_1, \pm e_2]$, $B = [\pm e_1, \pm e_2]$. This means that the maximal segment $B_1^n \cap \ell$ is contained in the $2$-dimensional cross-polytope $B_1^2\times \{0\}^{n-2}$. Rotating by $\pi/2$ and rescaling by $1/\sqrt{2}$, the problem then is equivalent to finding the maximal length section of the square $[-\frac12,\frac12]^2$ by lines at distance $t$ to the origin. This was done in Theorem 1 from \cite{zva} and the answer is exactly the right hand side of \eqref{eq:max-1-dim}. It remains to argue that in the other cases we do not get any longer sections than that.

\bigskip
\noindent
\emph{Case 2:} $|\{i, j, k, l \}| = 3$. By symmetry we can assume $i = k = 1$, $j = 2$, $l = 3$, that is $A = [\pm e_1, \pm e_2]$, $B = [\pm e_1, \pm e_3]$. By symmetry again, we can in fact assume that $A = [\pm e_1, e_2]$, $B = [\pm e_1, e_3]$. Depending on whether $A$ and $B$ share a vertex or not, we have two possibilities.

\bigskip
\noindent
\emph{Subcase 2.1:} $A = [e_1, e_2]$, $B = [e_1, e_3]$. Let $a$, $b$ be the endpoints of the segment $B_1^n \cap \ell$, say $\{a\} = A \cap \ell$ and $\{b\} = B \cap \ell$. Note that $|a-b| \leq \sqrt{2}$ (in fact $|a-b| \leq \max\{|a-e_1|,|b-e_1|\}$ as quickly follows from the sine rule in the triangle $\conv\{a,b,e_1\}$). Since $\ell$ is contained in the plane passing through $e_1, e_2, e_3$ and this plane is at distance $\frac{1}{\sqrt{3}}$ away from the origin, in this case we have $t \geq \frac{1}{\sqrt{3}}$. 

When $t \in [\frac{1}{\sqrt{3}},\frac{1}{\sqrt{2}}]$, the centred sphere of radius $t$, call it $S$, does not touch the edges of $B_1^n$ and we can find a segment with endpoints on $[e_1,e_2]$ and $[-e_1,-e_2]$ which is parallel to the edge $[e_1,-e_2]$ and is tangent to the sphere $S$; its length is $\sqrt{2}$ and the line passing through it, call it $\ell'$, is $t$ away from the origin. Then $|B_1^n \cap \ell'| = \sqrt{2} \geq |a-b| = |B_1^n \cap \ell|$, so the section $B_1^n \cap \ell$ is no longer than the section $B_1^n \cap \ell'$ (from Case 1). 

\begin{figure}[htb]
\begin{center}
\includegraphics[scale=1]{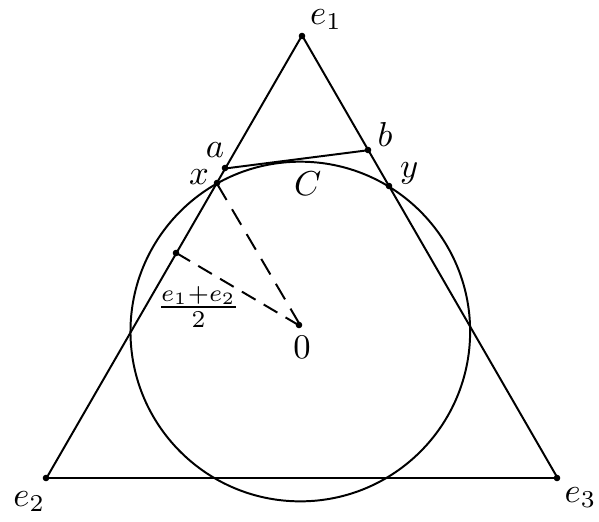}
\end{center}
\caption{Subcase 2.1 when $t > \frac{1}{\sqrt{2}}$.}\label{fig:subcase}
\end{figure}

Suppose $t \in (\frac{1}{\sqrt{2}},1]$. Then the intersection of the sphere $S$ and the $2$-dimensional face $\conv\{e_1,e_2,e_3\}$ consists of $3$ disjoint arcs. Let $C$ be the arc among those $3$ that is intersecting the edges $[e_1,e_2]$ and $[e_1,e_3]$, say at points $x$ and $y$ respectively. Since the segment $[a,b]$ is contained in that face, it is tangent to $C$ and thus $\max\{|a-e_1|, |b-e_1|\} \leq |x-e_1|$ (see Figure \ref{fig:subcase}). By Pythagoras' theorem for the triangle $\conv\{0, \frac{e_1+e_2}{2},x\}$, we have $\left(\frac{\sqrt{2}}{2} - |x-e_1|\right)^2 + \frac{1}{2} = t^2$, thus
\[
|a-b| \leq \max\{|a-e_1|, |b-e_1|\} \leq |x-e_1| = \frac{1}{\sqrt{2}} - \sqrt{t^2-\frac{1}{2}}.
\]
For $t \in (\frac{1}{\sqrt{2}},\frac{3}{4}]$, we have $\frac{1}{\sqrt{2}} - \sqrt{t^2-\frac{1}{2}} \leq t - \sqrt{t^2-\frac{1}{2}}$ and for $t \in [\frac{3}{4},1]$, we check that $ \frac{1}{\sqrt{2}} - \sqrt{t^2-\frac{1}{2}} \leq 2-2t$ (the left hand side is convex as a function of $t$ whereas the right hand side is linear, so it is enough to verify the inequality at the endpoints $t=\frac{3}{4}$ and $t = 1$). Therefore, $|a-b| = |B_1^n \cap \ell|$ is bounded by the right hand side of \eqref{eq:max-1-dim}, as desired.

\bigskip
\noindent
\emph{Subcase 2.2:} $A = [e_1, e_2]$, $B = [-e_1, e_3]$. Let $a \in A$ and $b \in B$ be the endpoints of the segment $B_1^n \cap \ell$, say $a = (1-\alpha)e_1 + \alpha e_2$ and $b = -(1-\beta)e_1 + \beta e_3$ for some $\alpha, \beta \in [0,1]$. Note that $\scal{a}{b} \leq 0$ and $|a|, |b| \leq 1$. 

\begin{lemma}\label{lm:small-dist}
Suppose $x$ and $y$ are distinct vectors in $\R^n$ such that $\scal{x}{y} \leq 0$ and each one is of length at most $1$. Then
\[
\frac{|x|^2|y|^2 - \scal{x}{y}^2}{|x-y|^2} \leq \frac{1}{2},
\]
that is, in view of Lemma \ref{lm:dist(l,0)}, the line passing through $x$ and $y$ is at most $\frac{1}{\sqrt{2}}$ away from the origin.
\end{lemma}
\begin{proof}
Estimates $|x|^2|y|^2 \leq |x||y|$, $\scal{x}{y} \leq 0$ and $\scal{x}{y}^2 \geq 0$ yield
\begin{align*}
|x-y|^2 - 2(|x|^2|y|^2 - \scal{x}{y}^2) &= |x|^2 + |y|^2 - 2|x|^2|y|^2 + 2\scal{x}{y}^2 - 2\scal{x}{y} \\
&\geq |x|^2 + |y|^2 - 2|x||y| = (|x|-|y|)^2 \geq 0.
\end{align*}
\end{proof}

\noindent
Lemma \ref{lm:small-dist} gives that $t \leq \frac{1}{\sqrt{2}}$. In view of \eqref{eq:max-1-dim}, it is then enough to show that 
\[
|a-b| = |B_1^n \cap \ell| \leq \frac{2}{t + \sqrt{1-t^2}}.
\] 
Equivalently, plugging in $t = \frac{\sqrt{|a|^2|b|^2 - \scal{a}{b}^2}}{|a-b|}$,
\[
\sqrt{|a|^2|b|^2 - \scal{a}{b}^2} + \sqrt{|a-b|^2 - |a|^2|b|^2 + \scal{a}{b}^2} \leq 2.
\]
Since
\begin{align*}
\sqrt{|a-b|^2 - |a|^2|b|^2 + \scal{a}{b}^2} &= \sqrt{-(1-|a|^2)(1-|b|^2) + (1-\scal{a}{b})^2} \\
&\leq \sqrt{(1-\scal{a}{b})^2} = 1 - \scal{a}{b}
\end{align*}
(in the last equality we used $|\scal{a}{b}| \leq |a||b| \leq 1$), it suffices to show
\begin{equation}\label{eq:goal-case22}
|a|^2|b|^2 - \scal{a}{b}^2 \leq (1 + \scal{a}{b})^2.
\end{equation}
Using coordinates, $a = (1-\alpha)e_1 + \alpha e_2$, $b = -(1-\beta)e_1 + \beta e_3$, we have $|a|^2 = \alpha^2 + (1-\alpha)^2$, $|b|^2 = \beta^2 + (1-\beta)^2$ and $\scal{a}{b} = -(1-\alpha)(1-\beta)$. Thus
\begin{align*}
(1 + \scal{a}{b})^2 - \big(|a|^2|b|^2 - \scal{a}{b}^2\big) &= \big(\alpha + \beta - \alpha\beta\big)^2 - \big(\alpha^2\beta^2 + \alpha^2(1-\beta)^2 + (1-\alpha)^2\beta^2\big) \\
&= \alpha^2 + \beta^2 + \alpha^2\beta^2 + 2\alpha\beta - 2\alpha^2\beta - 2\alpha\beta^2 \\
&\qquad\qquad\qquad\qquad-\big(3\alpha^2\beta^2 + \alpha^2 + \beta^2 - 2\alpha^2\beta - 2\alpha\beta^2 \big) \\
&= 2\alpha\beta - 2\alpha^2\beta^2 = 2\alpha\beta(1-\alpha\beta) \geq 0.
\end{align*}

\bigskip
\noindent
\emph{Case 3:} $|\{i, j, k, l \}| = 4$. By symmetry we can assume that $i = 1$, $j = 2$, $k = 3$, $l = 4$ and $A = [e_1, e_2]$, $B = [e_3, e_4]$. Let $a \in A$ and $b \in B$ be the endpoints of the segment $B_1^n \cap \ell$, say $a = (1-\alpha)e_1 + \alpha e_2$ and $b = (1-\beta)e_3 + \beta e_4$ for some $\alpha, \beta \in [0,1]$. Note that $\scal{a}{b} = 0$ and $|a|, |b| \leq 1$. Therefore, we can repeat verbatim the argument from Subcase 2.2 up to \eqref{eq:goal-case22}. Moreover, in this case, \eqref{eq:goal-case22} becomes $|a|^2|b|^2 \leq 1$, so it clearly holds. This finishes the proof.

\begin{remark}
The exact description of lines attaining maximum in \eqref{eq:max-1-dim} contained in a $2$-dimensional coordinate subspace can be found in the proof of Theorem 1 in \cite{zva}.
\end{remark}

\subsection{Minimal $1$-dimensional sections: Proof of Theorem \ref{thm:mim-1-dim}}

\subsubsection*{Case $t > \frac{1}{\sqrt{n}}$}

Plainly, there are lines $t$ away from the origin disjoint from $B_1^n$ (e.g. those parallel to facets)
and this explains the last case of \eqref{eq:min-1-dim}. 

\subsubsection*{Case $t \in [0,\frac{1}{\sqrt{n}}]$}

Let $\ell \in \mathscr{L}_t$ be a line at distance $t$ from the origin. The intersection $B_1^n\cap \ell$ is a segment, say $[a,b]$. When $t < \frac{1}{\sqrt{n}}$, the segment $[a,b]$ is not entirely contained in any of the facets of $B_1^n$ (otherwise its distance to the origin would be at least $\frac{1}{\sqrt{n}}$). We can thus assume that the endpoints $a$ and $b$ belong to two \emph{distinct} facets of $B_1^n$ (and not their intersection). By symmetry, we can assume that $a$ belongs to the facet $F_0 = \{x \in \R^n, \ x_1,\dots,x_n \geq 0, \sum_{i=1}^n |x_i| = 1\}$ and $b$ belongs to the facet $F_k = \{x \in \R^n, \ x_1,\dots,x_k \leq 0, x_{k+1},\dots, x_n \geq 0 , \sum_{i=1}^n |x_i| = 1\}$ for some $k \in \{1,\dots,n\}$ and $a \notin F_k$ and $b \notin F_0$. When $t = \frac{1}{\sqrt{n}}$, the segment $[a,b]$ may be entirely contained in one of the facets of $B_1^n$, say $F$ (if not, we proceed as earlier). Then the segment contains the centroid of $F$ (otherwise its distance to the origin would be larger than $\frac{1}{\sqrt{n}}$). The endpoints $a$, $b$ are thus contained in two distinct facets of $F$, say $F' \cap F$ and $F'' \cap F$ for some two other \emph{distinct} facets $F'$, $F''$ of $B_1^n$. Consequently, as earlier, we can say that $a \in F_0 \setminus F_k$ and $b \in F_k \setminus F_0$.

\bigskip
\noindent
\emph{Case 1:} $k = n$. The facets $F_0$ and $F_n$ are parallel at distance $\frac{2}{\sqrt{n}}$, hence 
\[
|B_1^n \cap \ell| = |a-b| \geq \frac{2}{\sqrt{n}}.\]

\bigskip
\noindent
\emph{Case 2:} $k \leq n-1$. Let $Q_k\colon \R^n\to\R^n$ be the orthogonal projection onto the $2$-dimensional subspace $G_k$ spanned by $u_k = \frac{\sum_{i=1}^k e_i}{k}$ and $v_k = \frac{\sum_{i=k+1}^n e_{i}}{n-k}$, that is 
\[
Q_k(x_1,\dots,x_n) = \left(\underbrace{\frac{\sum_{i=1}^k x_i}{k},\dots,\frac{\sum_{i=1}^k x_i}{k}}_{k},\underbrace{\frac{\sum_{i=k+1}^n x_i}{n-k},\dots,\frac{\sum_{i=k+1}^n x_i}{n-k}}_{n-k}\right).
\]
Note that the image of the cross-polytope is a diamond,
\[
Q_k(B_1^n) = Q_k(\conv\{\pm e_1,\dots,\pm e_n\}) = \conv\{\pm Q_ke_1,\dots, \pm Q_ke_n\} = \conv\{\pm u_k, \pm v_k\}.
\]
Call this diamond $D_k$.
Moreover, the facets $F_0$ and $F_k$ are mapped onto its two edges, $Q_k(F_0) = [u_k,v_k]$ and $Q_k(F_k) = [-u_k,v_k]$. Let $a' = Q_ka$, $b' = Q_kb$. Since $a \in F_0 \setminus F_k$ and $b \in F_k \setminus F_0$, we have that $a' \in [u_k,v_k)$ and $b' \in [-u_k,v_k)$. Given $t \in [0,\frac{1}{\sqrt{n}}]$, let 
\begin{align*}
m_k(t) = &\text{minimum length of sections of $D_k$ by lines which are $t$ away from the origin}\\
&\text{and intersect the edges $[u_k,v_k)$ and $[-u_k,v_k)$ (avoiding the vertex $v_k$).}
\end{align*}
Say the line passing through $a'$ and $b'$ is $t'$ away from the origin. Because it is the image of $\ell$ under the projection $Q_k$, we have $t' \leq t$. We thus get
\[
|B_1^n \cap \ell| = |a-b| \geq |a'-b'| \geq m_k(t') \geq m_k(t),
\]
where the last inequality follows from the fact that the length of sections of $D_k$ by parallel lines as a function of distance to the origin is even and concave on its support (Brunn's principle), hence nonincreasing. 

Combining Case 1 and 2 yields
\begin{equation}\label{eq:min-low-bd}
\min_{\ell \in \mathscr{L}_t} |B_1^n \cap \ell| \geq m(t)
\end{equation}
with
\[
m(t) = \min\left\{\frac{2}{\sqrt{n}}, \min_{1 \leq k \leq n-1} m_k(t) \right\}.
\]
We shall now find the minimum on the right hand side and then show that equality in \eqref{eq:min-low-bd} is in fact attained. Finding the function $m_k(t)$ boils down to solving a $2$-dimensional problem. For $k = 1,\dots,n-1$ and $t \in [0,\frac{1}{\sqrt{n}}]$, we have
\begin{equation}\label{eq:m_k(t)}
m_k(t) = 2\frac{1-t\sqrt{n-k}}{\sqrt{k}}
\end{equation}
(we defer its proof). Defining $m_n(t) = \frac{2}{\sqrt{n}}$, we can write $m(t) = \min_{1 \leq k \leq n} m_n(t)$. Note that $m_{k+1}(t) \geq m_k(t)$ if and only if $t \geq T_n(k)$ with $T_n(k) = \frac{\sqrt{(k+1)(n-k)}+\sqrt{k(n-k-1)}}{n(\sqrt{k}+\sqrt{k+1})}$. We check that
\[
0 < T_n(n-1) < T_n(n-2) < \dots < T_n(1) < T_n(0) = \frac{1}{\sqrt{n}}.
\]
As a result, for $t \in [0, T_n(n-1)]$, we have $m_1(t) \geq m_2(t) \geq \dots \geq m_n(t)$. For $t \in (T_n(k),T_n(k-1)]$ with $k = n-1,\dots,1$, we have $m_1(t) \geq \dots \geq m_{k-1}(t) \geq m_k(t) \leq m_{k+1}(t) \leq \dots \leq m_n(t)$, thus $m(t)$ is exactly the right hand side of \eqref{eq:min-1-dim}. 

It remains to show that this bound is attained, that is given $t$, there is a line $\ell$ which is $t$ away from the origin such that $|B_1^n \cap \ell| = m(t)$. To this end, first fix $t \in [0, T_n(n-1)]$. Let $\theta = t\sqrt{n(n-1)}$ (note $\theta \in [0,1)$), take $a = \theta\frac{\sum_{i=1}^{n-1} e_i}{n-1} + (1-\theta)\frac{\sum_{i=1}^n e_i}{n}$ and $b = a - \frac{2}{n}\sum_{i=1}^n e_i$. These are points on the boundary of $B_1^n$ (on the facets $F_0$ and $F_n$) and the line $\ell$ passing through them gives $|B_1^n \cap \ell| = |a-b| = \frac{2}{\sqrt{n}} = m(t)$. We check that $\ell$ is $t$ away from the origin (e.g. using Lemma \ref{lm:dist(l,0)}). Now fix $k \in \{1,\dots,n-1\}$ and $t \in (T_n(k),T_n(k-1)]$. Set $\theta = t\sqrt{n-k}$ (note that $\theta \in (0,1]$), $a = (1-\theta)u_k + \theta v_k$ and $b = -(1-\theta)u_k + \theta v_k$. These are boundary points (belonging to the facets $F_0$ and $F_k$), the line $\ell$ passing through them gives $|B_1^n \cap \ell| = |a-b| = 2(1-\theta)|u_k| = 2\frac{1-t\sqrt{n-k}}{\sqrt{k}} = m_k(t) = m(t)$ and we check that $\ell$ is $t$ away from the origin.

Finally, we are left with showing \eqref{eq:m_k(t)}. Fix $k \in \{1,\dots,n-1\}$ and $t \in [0,\frac{1}{\sqrt{n}}]$. Then \eqref{eq:m_k(t)} follows from the following elementary lemma, applied to $u = |u_k| = \frac{1}{\sqrt{k}}$ and $v = |v_k| = \frac{1}{\sqrt{n-k}}$.

\begin{lemma}\label{lm:isosceles}
Let $u, v > 0$, $t \in [0,\frac{uv}{\sqrt{u^2+v^2}}]$ and consider an isosceles triangle $\conv\{\pm ue_1,ve_2\}$. The minimal length of a segment with endpoints on the legs $[-ue_1,ve_2)$ and $[ue_1,ve_2)$ at distance $t$ from the origin is $2(v-t)\frac{u}{v}$ (attained if and only if the segment is parallel to the base).
\end{lemma}
\begin{figure}[htb]
\begin{center}
\includegraphics[scale=1]{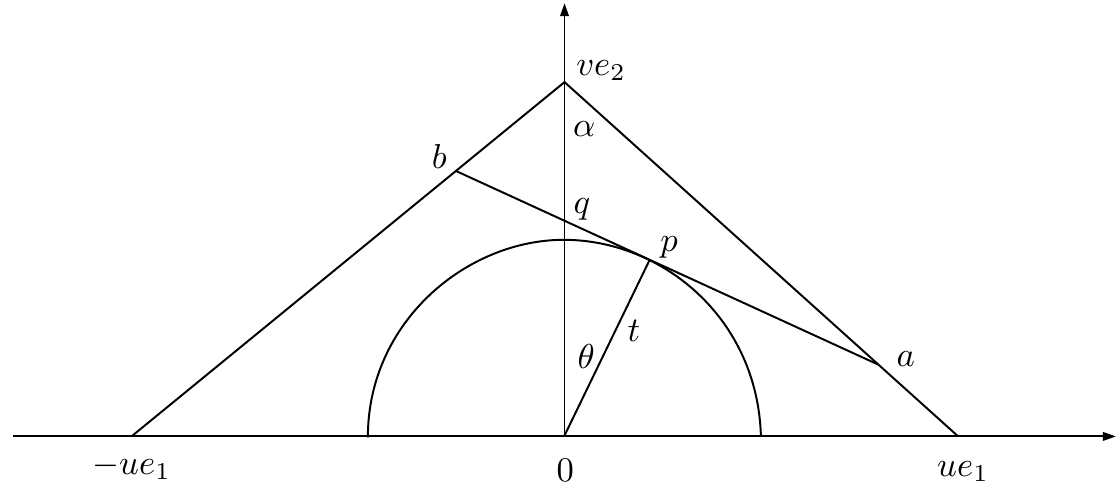}
\end{center}
\caption{Proof of Lemma \ref{lm:isosceles}.}\label{fig:tr}
\end{figure}
\begin{proof}
Let $a$ and $b$ be the endpoints on $[ue_1,ve_2)$ and $[-ue_1,ve_2)$ of such a segment which is tangent to the centred circle of radius $t$ at point, say $p$ (note that the circle is contained in the triangle). Let $\theta$ be the angle between the segments $[0,p]$, $[0,ve_2]$ and let $\alpha$ be the angle between the segments $[0,ve_2]$, $[ve_2,ue_1]$. Note that $\theta \in [0,\frac{\pi}{2}-\alpha)$. Let $q$ be the point of intersection of $[a,b]$ and $[0,ve_2]$ (see Figure \ref{fig:tr}). We have $|q-ve_2| = v - \frac{t}{\cos\theta}$ and by the sine rule, $|a-q| = \left( v - \frac{t}{\cos\theta}\right)\frac{\sin\alpha}{\sin(\frac{\pi}{2}+\alpha+\theta)}$, $|b-q| = \left( v - \frac{t}{\cos\theta}\right)\frac{\sin\alpha}{\sin(\frac{\pi}{2}+\alpha-\theta)}$. Thus,
\[
|a-b| = \left( v - \frac{t}{\cos\theta}\right)\sin\alpha\left(\frac{1}{\sin(\frac{\pi}{2}+\alpha+\theta)} + \frac{1}{\sin(\frac{\pi}{2}+\alpha-\theta)}\right),
\]
or, after simplifying,
\[
|a-b| = \frac{2(v\cos\theta-t)\sin\alpha\cos\alpha}{\cos^2\alpha+\cos^2\theta-1}.
\]
We want to show that $|a-b| \geq 2(v-t)\frac{u}{v} = 2(v-t)\tan\alpha$ with equality if and only if $\theta = 0$ (the segment $[a,b]$ is parallel to the base). Since the denominator is positive, $\cos^2\alpha+\cos^2\theta-1 > \cos^2\alpha + \cos^2(\frac{\pi}{2}-\alpha) - 1 > 0$, our desired inequality is equivalent to
\[
(v\cos\theta-t)\cos^2\alpha \geq (v-t)(\cos^2\alpha+\cos^2\theta-1)
\]
which becomes
\[
(1-\cos\theta)\Big((v-t)(1+\cos\theta) - v\cos^2\alpha\Big) \geq 0.
\]
When $\theta = 0$, we have equality. When $0 < \theta < \frac{\pi}{2}-\alpha$, we use $\cos\theta > \sin\alpha$, $t \leq \frac{uv}{\sqrt{u^2+v^2}} = v\sin\alpha$ and get 
\[
(v-t)(1+\cos\theta) - v\cos^2\alpha > (v - v\sin\alpha)(1 + \sin\alpha) - v\cos^2\alpha = 0,
\]
which proves the strict inequality, as desired.
\end{proof}

\begin{remark}\label{rem:t=1/sqrtn}
When $t = \frac{1}{\sqrt{n}}$, the set of lines $t$ away from the origin is larger than the set of all lines entirely contained in the hyperplane of a facet of $B_1^n$ and passing though its centroid. Asking about a minimal section of the latter amounts to asking about a minimal length section of a regular simplex by lines passing through its centroid. From \eqref{eq:min-1-dim}, we get a lower bound on this quantity by $2\left(1 - \sqrt{\frac{n-1}{n}}\right) = \frac{2}{n + \sqrt{n(n-1)}}$ for an $n-1$ dimensional regular simplex with edge length $\sqrt{2}$. This bound is not tight. For results about extremal central sections of a simplex see \cite{Brz, Dir, Webb}. This $1$-dimensional case however does not seem to have appeared anywhere explicitly. We provide results about extremal $1$-dimensional central sections in passing here. For concreteness, consider the $n-1$-dimensional regular simplex $S_n = \conv\{e_1,\ldots,e_n\}$ embedded in $\R^n$ and the set $L_0$ of all lines which lie in the hyperplane spanned by $S_n$ and pass through its centroid $z = \frac{\sum_{i=1}^n e_i}{n}$. With this notation, we have the following (folklore) theorems.
\end{remark}

\begin{theorem}
We have 
\[
\min_{\ell \in L_0} |S_n \cap \ell| = \frac{2\sqrt{2}}{n}.
\]
Moreover, the minimum is attained if and only if $\ell$ is parallel to one of the edges of $S_n$ (the line $\ell$ is along a direction $v = \frac{e_i - e_j}{\sqrt{2}}$ for $i \neq j$).  
\end{theorem}
\begin{proof}
Fix a unit vector $v$ in $\R^n$ such that $\sum v_i = 0$. If $\ell$ is a line in the direction of $v$, passing through $z$, then it can be checked that $|S_n \cap \ell| = \frac{1}{n}\left(\frac{1}{\max v_i} - \frac{1}{\min v_i}\right)$. Let $a = \max v_i$ and $b = -\min v_i$. Note that $a, b \in (0,1)$ and $a^2 + b^2 \leq 1$. By the harmonic-quadratic mean inequality, we have
\[
\frac{1}{\max v_i} - \frac{1}{\min v_i} = \frac{1}{a} + \frac{1}{b} \geq \frac{2}{\sqrt{\frac{a^2+b^2}{2}}} \geq 2\sqrt{2}
\]
and the first inequality holds if and only if $a = b$, whereas the second one if and only if $a^2 + b^2 = 1$, that is the equalities hold if and only if $a = b = \frac{1}{\sqrt{2}}$. This is the case if and only if $v$ is supported on exactly $2$ coordinates.
\end{proof}

\begin{theorem}
We have 
\[
\max_{\ell \in L_0} |S_n \cap \ell| = \sqrt{\frac{n}{n-1}}.
\]
Moreover, the maximum is attained if and only if $\ell$ passes through one of the vertices of $S_n$.  
\end{theorem}
\begin{proof}
It will be more convenient to use a different setup here. Fix $\ell \in L_0$. Then $S_n \cap \ell$ is a segment, say with endpoints $x$ and $y$ which are on the boundary of $S_n$. By symmetry, we can assume that $x \in \conv\{e_1,\ldots, e_{n-1}\}$, that is $x = (x_1,\ldots,x_{n-1},0)$ with nonnegative $x_i$ such that $\sum x_i = 1$. The other endpoint $y$ is given by the intersection of the line passing through $x$ and $z$ with the boundary of $S_n$. Every point on this line is of the form $z + t(z-x)$, $t \in \R$ and $y$ is given by the smallest positive $t$ such that at least one of the coordinates of this point is $0$, which gives $t = \frac{1}{n(\max_{1 \leq i \leq n-1} x_i)-1}$ and the length of the section is $|S_n \cap \ell| = |y - x| = (1+t)|z-x|$. Thus
\[
\max_{\ell \in L_0} |S_n \cap \ell| = 
\max_{x \in [0,1]^{n-1}, \sum_{i=1}^{n-1} x_i = 1}\left(1 + \frac{1}{n(\max_{1 \leq i \leq n-1} x_i) - 1}\right)\sqrt{\sum_{i=1}^{n-1} \left(\frac{1}{n}-x_i\right)^2 + \frac{1}{n^2}}.
\]
We fix a vector $x \in \R^{n-1}$ with nonnegative coordinates adding up to $1$. Then 
\[
\sum_{i=1}^{n-1} \left(\frac{1}{n} - x_i\right)^2 + \frac{1}{n^2} = \sum_{i=1}^{n-1} x_i^2 - \frac{1}{n}. 
\] 
Let $a = \max_{1 \leq i \leq n-1} x_i$. We have $\sum x_i^2 \leq a\sum x_i = a$ with equality if and only if the positive coordinates of $x$ are equal. Thus,
\begin{align*}
\left(1 + \frac{1}{n(\max  x_i) - 1}\right)\sqrt{\sum_{i=1}^{n-1} \left(\frac{1}{n}-x_i\right)^2 + \frac{1}{n^2}} &= \frac{na}{na-1}\sqrt{\sum_{i=1}^{n-1} x_i^2 - \frac{1}{n}} \\
&\leq \frac{na}{na-1}\sqrt{a - \frac{1}{n}} = \frac{1}{\sqrt{n}} \frac{na}{\sqrt{na-1}}.
\end{align*}
Note that $na \in [\frac{n}{n-1},n]$. The function $f(t) = \frac{t}{\sqrt{t-1}}$ is decreasing on $(1,2)$ and increasing on $(2,\infty)$. Moreover, $f(\frac{n}{n-1}) = f(n) = \frac{n}{\sqrt{n-1}}$. Consequently, $f(na) \leq \frac{n}{\sqrt{n-1}}$ with equality if and only if $a = \frac{1}{n-1}$ (equivalent to $x = \frac{\sum_{i=1}^{n-1} e_i}{n-1}$) or $a = 1$ (equivalent to $x = e_i$ for some $i \leq n-1$). We have thus obtained that 
\[
\left(1 + \frac{1}{n(\max  x_i) - 1}\right)\sqrt{\sum_{i=1}^{n-1} \left(\frac{1}{n}-x_i\right)^2 + \frac{1}{n^2}} \leq \frac{1}{\sqrt{n}} \frac{na}{\sqrt{na-1}} \leq \sqrt{\frac{n}{n-1}}
\]
with the same equality cases for both estimates. It remains to notice that they correspond to lines passing through a vertex of $S_n$.
\end{proof}

\subsection{Sections by hyperplanes at distance $t > \frac{1}{\sqrt{2}}$: Proof of Theorem~\ref{thm:max-hyp}}

Fix $t \in (\frac{1}{\sqrt{2}},1]$ and let $H$ be a hyperplane $t$ away from the origin. The case $t=1$ is clear. Let $t < 1$. Then $H$ intersects the interior of $B_1^n$ and separates exactly one of its vertices from the origin (it cannot separate two or more vertices because the edges of $B_1^n$ are $\frac{1}{\sqrt{2}}$ away from the origin). By symmetry, we can assume that the separated vertex is $e_1$. There are $2n-2$ edges coming out of $e_1$: $[e_1,\varepsilon e_k]$, $k = 2,\dots,n$, $\varepsilon \in \{-1,1\}$. Suppose $H$ intersects the edge $[e_1,\varepsilon e_k]$ at a point $v_{\varepsilon e_k} = \lambda_{\varepsilon, k} e_1 + (1-\lambda_{\varepsilon, k})\varepsilon e_k$. Let $a$ be a unit vector normal to $H$, so that $H$ is given by the equation $\scal{a}{x} = t$. We remark that
\begin{equation}\label{eq:dom-a}
a_1 \in (t,1], \qquad a_k^2 \leq 1 - a_1^2 < 1 - t^2 < t^2 < a_1^2, \qquad \text{for every $k = 2,\dots, n$}
\end{equation}
(since $H$ separates $e_1$, we have $a_1 = \scal{e_1}{a} > t$; then $a_1^2 + a_k^2 \leq \sum_{i=1}^n a_i^2 = 1$).
We find that $\lambda_{\varepsilon, k} = \frac{t - \varepsilon a_k}{a_1 - \varepsilon a_k}$, thus $v_{\varepsilon e_k} = \frac{t - \varepsilon a_k}{a_1-\varepsilon a_k}e_1 + \frac{a_1-t}{a_1-\varepsilon a_k}\varepsilon e_k$. Let $S = B_1^n \cap \{x \in \R^n, \ \scal{x}{a} \geq t\}$ be the chopped-off part of $B_1^n$ by the hyperplane $H$. Note that $S$ is the union of $2^{n-2}$ simplices $S_{\bar \varepsilon}$ with pairwise disjoint interiors,
\[
S_{\bar \varepsilon} = \conv\{e_1, v_{\varepsilon_2 e_2}, \dots, v_{\varepsilon_{n-1} e_{n-1}}, v_{-e_n}, v_{e_n}\}, \qquad \bar \varepsilon = (\varepsilon_2,\ldots,\varepsilon_{n-1}) \in \{-1,1\}^{n-2}.
\] 
We can find their volume by evaluating appropriate determinants,
\begin{align*}
|S_{\bar \varepsilon}| &= \frac{1}{n!}|\det[ v_{\varepsilon_2 e_2} - e_1, v_{\varepsilon_3 e_3}, \dots, v_{\varepsilon_{n-1} e_{n-1}} - e_1, v_{-e_n} - e_1, v_{e_n} - e_1 ]|\\
&= \frac{1}{n!}|\det\left[\begin{matrix}
\frac{t-a_1}{a_1-\varepsilon_{2}a_2} & \frac{t-a_1}{a_1-\varepsilon_{3}a_3} & \cdots & \frac{t-a_1}{a_1-\varepsilon_{n-1}a_{n-1}} & \frac{t-a_1}{a_1+a_n} & \frac{t-a_1}{a_1-a_n}\\
\frac{\varepsilon_2(a_1-t)}{a_1-\varepsilon_2a_2} & 0 & \cdots & 0 & 0 & 0\\
0 & \frac{\varepsilon_3(a_1-t)}{a_1-\varepsilon_3a_3} & \cdots & 0 & 0 & 0\\
\vdots & \vdots & & \vdots & \vdots & \vdots\\
0 & 0 & \cdots & \frac{\varepsilon_{n-1}(a_1-t)}{a_1-\varepsilon_{n-1}a_{n-1}} & 0 & 0\\
0 & 0 & \cdots & 0 & \frac{t-a_1}{a_1+a_n} & \frac{a_1-t}{a_1-a_n}
\end{matrix}\right]|\\
&=\frac{2(a_1-t)^n}{n!(a_1^2-a_n^2)\prod_{i=2}^{n-1}(a_1-\varepsilon_ia_i)}.
\end{align*}
Thus, the volume of $S$ is
\begin{align*}
|S|=\sum_{\bar\varepsilon\in\{-1,1\}^{n-2}}|S_{\bar\varepsilon}| &= \frac{2(a_1-t)^n}{n!(a_1^2-a_n^2)}\sum_{\bar\varepsilon\in\{-1,1\}^{n-2}}\prod_{i=2}^{n-1}\frac{1}{a_1-\varepsilon_ia_i}\\
&=\frac{2(a_1-t)^n}{n!(a_1^2-a_n^2)}\prod_{i=2}^{n-1}\left(\frac{1}{a_1-a_i}+\frac{1}{a_1+a_i}\right)\\
&=\frac{2^{n-1}(a_1-t)^na_1^{n-2}}{n!\prod_{i=2}^{n}(a_1^2-a_i^2)}.
\end{align*}
On the other hand,
\begin{equation}\label{eq:volS}
|S| = \frac{1}{n}|B_1^n \cap H|\cdot\mathrm{dist}(e_1,H) = \frac{1}{n}|B_1^n \cap H|\cdot (a_1-t).
\end{equation}
which gives
\[
|B_1^n \cap H| = \frac{2^{n-1}}{(n-1)!}\frac{a_1^{n-2}(a_1-t)^{n-1}}{\prod_{i=2}^n(a_1^2-a_i^2)}.
\]
To finish the proof, it remains to show that for every $t \in (\frac{1}{\sqrt{2}},1)$ and for every unit vector $a \in \R^n$ such that $a_1 > t$, we have
\begin{equation}\label{eq:minH}
a_1^{n-2}(a_1-t)^{n-1}\leq (1-t)^{n-1}\prod_{i=2}^n(a_1^2-a_i^2)
\end{equation}
with equality if and only if $a = e_1$ (recall \eqref{eq:dom-a}). 

\bigskip
\noindent
\emph{Case 1:} $n=3$. Using $a_1^2+a_2^2+a_3^2 = 1$, we get $\prod_{i=2}^n(a_1^2-a_i^2) = (a_1^2-a_2^2)(a_1^2-a_3^2) = a_1^4-a_1^2(a_2^2+a_3^2) + a_2^2a_3^2 \geq a_1^4 - a_1^2(a_2^2+a_3^2) = 2a_1^4-a_1^2 = a_1^2(2a_1^2-1)$. It is then enough to show that for every $x \in (t,1)$, we have
\[
x(x-t)^2 < (1-t)^2x^2(2x^2-1).
\]
Consider the function $f(x) = (1-t)^2x(2x^2-1)-(x-t)^2$. Since $f''(x) = 12(1-t)^2x - 2 < 12\left(1 - \frac{1}{\sqrt{2}}\right)^2 - 2 < 0$, we have that $f$ is strictly concave on $[t,1]$. Since $f(t) > 0$ and $f(1) = 0$, we get that $f(x) > 0$ for $x \in (t,1)$, which finishes the proof of \eqref{eq:minH} when $n = 3$.

\bigskip
\noindent
\emph{Case 2:} $n\geq 4$. Since $a_i^2 \leq 1-a_1^2$, for every $i \geq 2$, we obtain $\prod_{i=2}^n(a_1^2 - a_i^2) \geq (2a_1^2-1)^{n-1}$. It is then enough to show that for every $x \in (t,1)$, we have
\[
x^{n-2}(x-t)^{n-1} < (1-t)^{n-1}(2x^2-1)^{n-1}.
\]
Consider the function $f(x) = (1-t)(2x^2-1)x^{-\frac{n-2}{n-1}} - (x-t)$. Let $\alpha = \frac{n-2}{n-1}$. Note that $\alpha \in [\frac{2}{3},1)$. For $x \in [t,1]$, we thus have
\begin{align*}
\frac{x^{\alpha+2}}{1-t}f''(x) = 2(2-\alpha)(1-\alpha)x^{2} - \alpha(\alpha+1) &\leq 2(2-\alpha)(1-\alpha) - \alpha(\alpha+1) \\
&= \alpha^2-7\alpha + 4 < 0.
\end{align*}
Consequently, $f$ is strictly concave on $[t,1]$. Since $f(t) > 0$ and $f(1) = 0$, we get that $f(x) > 0$ for $x \in (t,1)$, which finishes the proof of \eqref{eq:minH} when $n \geq 4$. This completes the proof of \eqref{eq:max-hyp}.

\subsection{Sections by slabs of width $t > \frac{1}{\sqrt{2}}$: Proof of Theorem~\ref{thm:min-slab}}

When $t = 1$, the theorem is clear. Let $a$ be a unit vector in $\R^n$ and fix $t \in (\frac{1}{\sqrt{2}},1)$. Note that the complement (in $B_1^n$) of the intersection of $B_1^n$ and the slab $\{x \in \R^n, \ |\scal{x}{a}| \leq t\}$ is exactly ``twice'' the chopped-off part of $B_1^n$ by the hyperplane $H = \{x \in \R^n, \ \scal{x}{a} = t\}$ (say it separates vertex $e_1$), as we analysed it in the previous section. In particular, from \eqref{eq:volS}, we immediately get
\[
|B_1^n \cap \{x \in \R^n, \ |\scal{x}{a}| \leq t\}| = |B_1^n| - 2|S| = |B_1^n| - \frac{2}{n}|B_1^n \cap H|\cdot (a_1-t)
\]
We already know that $|B_1^n \cap H|$ is maximised only at $a = e_1$ and, plainly, the same holds for $a_1-t$. This immediately gives \eqref{eq:min-slab}.

\end{document}